\documentclass{amsart}

\newtheorem{theorem}{Theorem}[section]
\newtheorem{lemma}[theorem]{Lemma}

\theoremstyle{definition}

\newtheorem{example}[theorem]{Example}

\newtheorem{proposition}[theorem]{Proposition}
\newtheorem{corollary}[theorem]{Corollary}

\theoremstyle{remark}

\numberwithin{equation}{section}

\begin{document}

\title[ Evolution  of the first eigenvalue of weighted $p$-Laplacian]{ Evolution  of the first eigenvalue of weighted $p$-Laplacian along the 
Ricci-Bourguignon flow }

\author{Shahroud Azami}
\address{Department of Mathematics, Faculty of Sciences, Imam Khomeini International University, Qazvin, Iran. }

\email{azami@sci.ikiu.ac.ir}






\begin{abstract}
Let $M$ be an $n$-dimensional closed Riemannian manifold with metric $g$, $d\mu=e^{-\phi(x)}d\nu$ be the weighted measure and $\Delta_{p,\phi}$ be the weighted $p$-Laplacian.
 In this article we will investigate monotonicity for
the first eigenvalue problem of the weighted $p$-Laplace operator acting on the
space of functions along the Ricci-Bourguignon  flow on closed Riemannian manifolds. We find the first  variation
formula for the eigenvalues of  the weighted $p$-Laplacian on a closed Riemannian manifold
evolving by the Ricci-Bourguignon  flow  and we obtain  various monotonic quantities. At the end we find some applications in $2$-dimensional  and  $3$-dimensional manifolds and give an example.
\end{abstract}

\maketitle
\section{Introduction}
A smooth metric measure space is a triple $(M,g, d\mu)$, where $g$ is a metric, $d\mu=e^{-\phi(x)}d\nu$ is the weighted volume measure on $(M,g)$ related to function $\phi\in C^{\infty}(M)$ and  $d\nu$ is  the Riemannian volume measure. Such spaces have been used more widely in the work of mathematicians, for instance, Perelman used it in \cite{GP}. Let $M$ be an $n$-dimensional closed Riemannian manifold with metric $g$.

Over the last few years  the  geometric flows  as  the Ricci-Bourguignon flow  have been a topic of active
research interest in both mathematics and physics. A geometric flow is an evolution of a geometric structure under
a differential equation related to a functional on a
manifold, usually associated with some curvature. The  family
${g(t)}$ of Riemannian metrics on $M$
 is called a   Ricci-Bourguignon flow when it satisfies the equations\\
\begin{equation}\label{rb}
\frac{d}{dt}g(t)=-2Ric(g(t))+2\rho R(g(t))g(t)=-2(Ric-\rho Rg) ,\\\
\end{equation}
with the initial condition
\begin{equation*}
g(0)=g_{0}\\
\end{equation*}
where $Ric$  is the Ricci tensor of $g(t)$, $R$ is the scalar curvature and $\rho$ is a real constant.  When $\rho=0$, $\rho=\frac{1}{2}$, $\rho=\frac{1}{n}$  and $\rho=\frac{1}{2(n-1)}$,  the tensor $Ric-\rho Rg$  corresponds to  the Ricci tensor, Einstein tensor, the traceless Ricci tensor and Schouten tensor respectively. In fact the Ricci-Bourguignon  flow is a system of partial differential equations  which was introduced   by Bourguignon for the first time in
 1981 (see \cite{JPB}).
Short time
existence and uniqueness for
solution to the Ricci-Bourguignon flow on $[0,T)$ have been shown  by Catino and et'al in  ~\cite{GC} for $\rho<\frac{1}{2(n-1)}$.
When  $\rho=0$,
the Ricci-Bourguignon flow is the Ricci flow.\\
 Let $f:M\to \mathbb{R}$, $f\in W^{1,p}(M)$ where  $W^{1,p}(M)$ is the Sobolev space.
For, $p\in [1,+\infty)$, the $p$-Laplacian of  $f$ defined as
\begin{equation}
\Delta_{p}f=div(|\nabla f|^{p-2}\nabla f)=|\nabla f|^{p-2}\Delta f+(p-2)|\nabla f|^{p-4}(Hess f)(\nabla f, \nabla f).
\end{equation}
The Witten-Laplacian is defined by $\Delta_{\phi}=\Delta-\nabla \phi . \nabla$, which is a symmetric diffusion operator on $L^{2}(M,\mu)$ and is  self-adjoint.
Now, for $p\in[1,+\infty)$ and any smooth function $f$ on $M$, we define the weighted $p$-Laplacian on $M$ by
\begin{equation}
\Delta_{p,\phi}f=e^{\phi}div \left( e^{-\phi}|\nabla f|^{p-2}\nabla f\right)=\Delta_{p}f-|\nabla f|^{p-2}\nabla \phi.\nabla f.
\end{equation}
In
the  weighted $p$-Laplacian when $\phi$ is a constant function, the weighted $p$-Laplace operator is just the  $p$-Laplace operator and  when $p=2$, the weighted $p$-Laplace operator is  the Witten-Laplace operator.\\

Let  $\Lambda$ satisfies in $-\Delta_{p,\phi}f =\Lambda | f|^{p-2}f,$ for some $f\in W^{1,p}(M)$, in this case we say $\Lambda$ is an eigenvalue of the weighted $p$-Laplacian  $\Delta_{p,\phi}$ at time $t\in[0,T)$. Notice that   $\Lambda$  equivalently satisfies in
\begin{equation}
-\int_{M}f \Delta_{p,\phi}f d\mu=\Lambda\int_{M} | f|^{p}d\mu,
\end{equation}
where $d\mu=e^{-\phi(x)}d\nu$ and $d\nu$ is the Riemannian volume measure and using the integration by parts it results that
\begin{equation}
\int_{M}|\nabla f|^{p}\,d\mu=\Lambda\int_{M} | f|^{p}d\mu,
\end{equation}
in above equation, $f(x,t)$ called eigenfunction corresponding to eigenvalue $\Lambda(t)$. The first non-zero eigenvalue $\lambda(t)=\lambda(M, g(t),d\mu)$ is defined as follows
\begin{equation}
\lambda(t)=\mathop {\inf} \limits_{0\neq f\in W_{0}^{1,p}(M)} \left\{ \int_{M}|\nabla f|^{p}d\mu:\,\,\,\,\int_{M}|f|^{p}d\mu=1 \right\},
\end{equation}
where $W_{0}^{1,p}(M)$ is the completion of $C_{0}^{\infty}(M)$ with respect Sobolev norm
\begin{equation}
||f||_{W^{1,p}}=\left(\int_{M}|f|^{p}d\mu+\int_{M}|\nabla f|^{p}d\mu \right)^{\frac{1}{p}}.
\end{equation}
The eigenvalue problem for  weighted $p$-Laplacian has been extensively  studied in the literature  \cite{LFW, LFW1}. \\

The problem of  monotonicity  of the  eigenvalue of  geometric  operator is a known and an intrinsic  problem. Recently many mathematicians study properties of evolution  of the  eigenvalue of  geometric  operators, for instance, Laplace, $p$-Laplace, Witten-Laplace, along  various geometric flows, for example,  Yamabe flow, Ricci flow, Ricci-Bourguignon flow, Ricci-harmonic  flow and mean curvature flow. The main study of evolution of the  eigenvalue of  geometric  operator along the geometric flow  began  when  Perelman  in \cite{GP}  showed
 that the first eigenvalue of the geometric  operator $-4\Delta+R$ is nondecreasing along the Ricci flow, where   $R$ is  scalar curvature.

Then Cao \cite{XC2}  and  Zeng and et'al \cite{FZ}   extended the geometric operator $-4\Delta+R$ to the operator $-\Delta+cR$ on closed Riemannian manifolds, and  investigated the monotonicity of eigenvalues of the operator $-\Delta+cR$  under the Ricci flow and the Ricci-Bourguignon flow, respectively.

Author in \cite{Az2}  studied the monotonicity  of the  first eigenvalue of Witten-Laplace operator $-\Delta_{\phi}$ along the Ricci-Bourguignon flow with some assumptions and in
 \cite{Az1} investigate  the evolution  for the first eigenvalue of  $p$-Laplacian along the Yamabe flow.

In  \cite{SF} and \cite{FSW}
have been  studied the evolution for the first eigenvalue of geometric operator $-\Delta_{\phi}+\frac{R}{2}$ along the Yamabe flow and  the Ricci flow, respectively.
For the other recent research in this subject, see \cite{LF, CY, JY}.\\

Motivated by the described  above works, in this paper we will study the evolution of  the first  eigenvalue of the weighted $p$-Laplace operator whose metric satisfying   the   Ricci-Bourguignon  flow  (\ref{rb}) and $\phi$ evolves by $\frac{\partial \phi}{\partial t}=\Delta \phi$ that is $(M^{n}, g(t),\phi(t))$ satisfying in following system
\begin{equation}\label{rb0}
\begin{cases}
\frac{d}{dt}g(t)=-2Ric(g(t))+2\rho R(g(t))g(t)=-2(Ric-\rho Rg),& g(0)=g_{0},\\
\frac{\partial \phi}{\partial t}=\Delta \phi&\phi(0)=\phi_{0},
\end{cases}
\end{equation}
where $\Delta$ is Laplace operator of metric $g(t)$.
\section{Preliminaries}
In this  section, we will discuss about the  differentiable of first nonzero eigenvalue  and its corresponding eigenfunction of the weighted $p$-Laplacian $\Delta_{p,\phi}$ along the  flow (\ref{rb0}).  Let $M$ be a closed oriented Riemannian $n$-manifold and  $(M, g(t), \phi(t))$ be a smooth solution of the evolution equations system (\ref{rb0})  for $t\in[0,T)$.\\

  In what follows we assume that  $\lambda(t)$ exists and is $C^{1}$-differentiable under the flow (\ref{rb0}) in the given interval $t\in[0,T)$. The  first nonzero eigenvalue of weighted $p$-Laplacian and its corresponding eigenfunction are not known to be $C^{1}$-differentiable. For this reason, we apply techniques  of Cao \cite{XC1} and Wu \cite{JY}  to study the evolution and monotonicity of $\lambda(t)=\lambda(t, f(t))$, where $\lambda(t, f(t))$ and $f(t)$ are assumed to be smooth. For this end, we assume that at time $t_{0}$, $f_{0}=f(t_{0})$ is the eigenfunction for the first eigenvalue $\lambda(t_{0})$  of   $\Delta_{p,\phi}$. Then we have
\begin{equation}
\int_{M}| f(t_{0})|^{p}\,d\mu_{g(t_{0})}=1.
\end{equation}
Suppose that
\begin{equation}
h(t):=f_{0}\left[\frac{\det(g_{ij}(t_{0}))}{\det(g_{ij}(t))}\right]^{\frac{1}{2(p-1)}},
\end{equation}
along the Ricci-Bourguignon flow $g(t)$. We assume that
\begin{equation}
f(t)=\frac{h(t)}{(\int_{M}|h(t)|^{p}d\mu)^\frac{1}{p}},
\end{equation}
which $f(t)$ is smooth function along the Ricci-Bourguignon flow, satisfied in $\int_{M}|f|^{p}d\mu=1$ and at time $t_{0}$, f is the eigenfunction for $\lambda$ of $\Delta_{p,\phi}$. Therefore if  $\int_{M}|f|^{p}d\mu=1$ and
\begin{equation}
\lambda(t, f(t))=-\int_{M}f \Delta_{p,\phi}fd\mu,
\end{equation}
then $\lambda(t_{0}, f(t_{0}))=\lambda(t_{0})$.
\section{Variation of $\lambda(t)$}
In this section, we will find  some useful evolution formulas for $\lambda(t)$ along the  flow (\ref{rb0}). Now, we  first recall some evolution of geometric structure along the   Ricci-Bourguignon flow and then give a useful proposition  about the variation of eigenvalues of  the weighted $p$-Laplacian under the  flow (\ref{rb0}). From \cite{GC} we have
\begin{lemma}\label{l}
Under the Ricci-Bourguignon flow equation (\ref{rb}), we get
\begin{enumerate}
  \item $\frac{\partial}{\partial t}g^{ij}=2(R^{ij}-\rho R g^{ij})$,\\
  \item$\frac{\partial}{\partial t}(d\nu)=(n\rho-1)Rd\nu$,\\
  \item $\frac{\partial}{\partial t}(d\mu)=(-\phi_{t}+(n\rho-1)R)d\mu$,\\
  \item$ \frac{\partial}{\partial t}(\Gamma_{ij}^{k})=-\nabla_{j}R_{i}^{k}-\nabla_{i}R_{j}^{k}+\nabla^{k}R_{ij}
+\rho(\nabla_{j}R\delta_{i}^{k}+\nabla_{i}R\delta_{j}^{k}-\nabla^{k}Rg_{ij})$,\\
  \item $\frac{\partial}{\partial t}R=[1-2(n-1)\rho]\Delta R+2|Ric|^{2}-2\rho R^{2}$,
\end{enumerate}
where $R$ is scalar curvature.
\end{lemma}
\begin{lemma}\label{l1}
Let $(M,g(t),\phi(t)),\,\,\,t\in[0,T)$ be a solution to the  flow (\ref{rb0}) on a closed oriented Riemannain manifold  for $\rho<\frac{1}{2(n-1)}$. Let $f\in C^{\infty}(M)$ be a smooth function on $(M,g(t))$. Then we have  the following evolutions:
\begin{eqnarray}
  \frac{\partial}{\partial t}|\nabla f|^{2}&=&2R^{ij}\nabla_{i}f\nabla_{j}f-2\rho R |\nabla f|^{2}+2g^{ij}\nabla_{i}f\nabla_{j}f_{t},\label{el1}\\
\frac{\partial}{\partial t}|\nabla f|^{p-2}&=&(p-2)|\nabla f|^{p-4}\{R^{ij}\nabla_{i}f\nabla_{j}f-\rho R |\nabla f|^{2}+g^{ij}\nabla_{i}f\nabla_{j}f_{t} \},\label{el2}\\
\frac{\partial}{\partial t}(\Delta f)&=&2R^{ij}\nabla_{i}\nabla_{j}f+\Delta f_{t}-2\rho R\Delta f-(2-n)\rho \nabla^{k}R\nabla_{k}f,\label{el3}\\
\frac{\partial}{\partial t}(\Delta_{p} f)&=&2R^{ij}\nabla_{i}(Z\nabla_{j}f)-2\rho R\Delta_{p}f+g^{ij}\nabla_{i}(Z_{t}\nabla_{j}f)\label{el4}\\\nonumber
&&+g^{ij}\nabla_{i}(Z\nabla_{j}f_{t})+\rho(n-2) Zg^{ij}\nabla_{i}R\nabla_{j}f,\\
\frac{\partial}{\partial t}(\Delta_{p,\phi} f)&=&2R^{ij}\nabla_{i}(Z\nabla_{j}f)+g^{ij}\nabla_{i}(Z_{t}\nabla_{j}f)+g^{ij}\nabla_{i}(Z\nabla_{j}f_{t})\label{el5}\\\nonumber
&&-2\rho R\Delta_{p,\phi} f+\rho(n-2)Zg^{ij}\nabla_{i}R\nabla_{j}f -Z_{t}\nabla\phi.\nabla f\\\nonumber
&& -2ZR^{ij}\nabla_{i}\phi\nabla_{j}f-Z\nabla\phi_{t}.\nabla f-Z\nabla\phi.\nabla f_{t},
\end{eqnarray}
where $Z:=|\nabla f|^{p-2}$ and $f_{t}=\frac{\partial f}{\partial t}$.
\end{lemma}
\begin{proof}
By direct computation in local coordinates we have
\begin{eqnarray*}
\frac{\partial}{\partial t}|\nabla f|^{2}&=&\frac{\partial}{\partial t}(g^{ij}\nabla_{i}f\nabla_{j}f)\\
&=&\frac{\partial g^{ij}}{\partial t}\nabla_{i}f\nabla_{j}f+2g^{ij}\nabla_{i}f\nabla_{j}f_{t}
\\&=&2R^{ij}\nabla_{i}f\nabla_{j}f-2\rho R |\nabla f|^{2}+2g^{ij}\nabla_{i}f\nabla_{j}f_{t},
\end{eqnarray*}
which exactly (\ref{el1}). We prove (\ref{el2}) by using (\ref{el1}) as follows
\begin{eqnarray*}
\frac{\partial}{\partial t}|\nabla f|^{p-2}&=&\frac{\partial}{\partial t}(|\nabla f|^{2})^{\frac{p-2}{2}}\\
&=&\frac{p-2}{2}(|\nabla f|^{2})^{\frac{p-4}{2}}\frac{\partial}{\partial t}(|\nabla f|^{2})
\\&=&\frac{p-2}{2}|\nabla f|^{p-4}\left\{2R^{ij}\nabla_{i}f\nabla_{j}f-2\rho R |\nabla f|^{2}+2g^{ij}\nabla_{i}f\nabla_{j}f_{t}\right\}\\
&=&(p-2)|\nabla f|^{p-4}\left\{R^{ij}\nabla_{i}f\nabla_{j}f-\rho R |\nabla f|^{2}+g^{ij}\nabla_{i}f\nabla_{j}f_{t}\right\},
\end{eqnarray*}
which is (\ref{el2}). Now previous Lemma and $2\nabla^{i}R_{ij}=\nabla_{j}R$ result that
\begin{eqnarray*}
\frac{\partial}{\partial t}(\Delta f)&=&\frac{\partial}{\partial t}[g^{ij}(\frac{\partial^{2}f}{\partial x^{i}\partial x^{j}}-\Gamma_{ij}^{k}\frac{\partial f}{\partial x^{k}})]\\
&=&\frac{\partial g^{ij}}{\partial t}(\frac{\partial^{2}f}{\partial x^{i}\partial x^{j}}-\Gamma_{ij}^{k}\frac{\partial f}{\partial x^{k}})
+g^{ij}(\frac{\partial^{2}f_{t}}{\partial x^{i}\partial x^{j}}-\Gamma_{ij}^{k}\frac{\partial f_{t}}{\partial x^{k}})
-g^{ij}\frac{\partial}{\partial t}(\Gamma_{ij}^{k})\frac{\partial f}{\partial x^{k}}\\
&=&2R^{ij}\nabla_{i}\nabla_{j}f-2\rho R \Delta f+\Delta f_{t}-g^{ij}\left\{ -\nabla_{j}R_{i}^{k}-\nabla_{i}R_{j}^{k}+\nabla^{k}R_{ij}\right\}\nabla_{k}f\\
&&-g^{ij}\rho(\nabla_{j}R\delta_{i}^{k}+\nabla_{i}R\delta_{j}^{k}-\nabla^{k}Rg_{ij})\nabla_{k}f\\
&=&2R^{ij}\nabla_{i}\nabla_{j}f+\Delta f_{t}-2\rho R\Delta f-(2-n)\rho \nabla^{k}R\nabla_{k}f.
\end{eqnarray*}
Let $Z=|\nabla f|^{p-2}$ we get
\begin{eqnarray*}
\frac{\partial}{\partial t}(\Delta_{p}f)&=&\frac{\partial}{\partial t}\big(div (|\nabla f|^{p-2}\nabla f)\big)=
\frac{\partial}{\partial t}\big(g^{ij}\nabla_{i} (Z\nabla_{j} f)\big)\\
&=&\frac{\partial}{\partial t}\big(g^{ij}\nabla_{i}Z\nabla_{j} f+g^{ij}Z\nabla_{i}\nabla_{j} f\big)\\
&=&\frac{\partial g^{ij}}{\partial t}\nabla_{i}Z\nabla_{j} f+g^{ij}\nabla_{i}Z_{t}\nabla_{j} f+g^{ij}\nabla_{i}Z\nabla_{j} f_{t}+Z_{t}\Delta f+Z\frac{\partial}{\partial t}(\Delta f)\\
&=&2R^{ij}\nabla_{i}Z\nabla_{j} f-2\rho Rg^{ij}\nabla_{i}Z\nabla_{j} f+g^{ij}\nabla_{i}Z_{t}\nabla_{j} f+g^{ij}\nabla_{i}Z\nabla_{j} f_{t}+Z_{t}\Delta f\\&&+Z\{2R^{ij}\nabla_{i}\nabla_{j}f+\Delta f_{t}-2\rho R\Delta f-(2-n)\rho \nabla^{k}R\nabla_{k}f\}\\
&=&2R^{ij}\nabla_{i}(Z\nabla_{j}f)-2\rho R\Delta_{p}f+g^{ij}\nabla_{i}(Z_{t}\nabla_{j}f)\\
&&+g^{ij}\nabla_{i}(Z\nabla_{j}f_{t})+\rho(n-2) Zg^{ij}\nabla_{i}R\nabla_{j}f.
\end{eqnarray*}
We have $\Delta_{p,\phi}f=\Delta_{p}f-|\nabla f|^{p-2}\nabla\phi.\nabla f$. Taking derivative with respect to time of both sides of last equation and  (\ref{el4}) imply that
\begin{eqnarray*}
\frac{\partial}{\partial t}(\Delta_{p,\phi} f)&=&\frac{\partial}{\partial t}(\Delta_{p} f)-Z\frac{\partial g^{ij}}{\partial t}\nabla_{i}\phi\nabla_{j}f-Z_{t}g^{ij}\nabla_{i}\phi\nabla_{j}f-Zg^{ij}\nabla_{i}\phi_{t}\nabla_{j}f\\&&-Zg^{ij}\nabla_{i}\phi\nabla_{j}f_{t}\\
&=&2R^{ij}\nabla_{i}(Z\nabla_{j}f)-2\rho R\Delta_{p}f+g^{ij}\nabla_{i}(Z_{t}\nabla_{j}f)+g^{ij}\nabla_{i}(Z\nabla_{j}f_{t})\\
&&+\rho(n-2) Zg^{ij}\nabla_{i}R\nabla_{j}f
-2ZR^{ij}\nabla_{i}\phi\nabla_{j}f+2\rho ZRg^{ij}\nabla_{i}\phi\nabla_{j}f\\&&-Z_{t}g^{ij}\nabla_{i}\phi\nabla_{j}f-Zg^{ij}\nabla_{i}\phi_{t}\nabla_{j}f-Zg^{ij}\nabla_{i}\phi\nabla_{j}f_{t},
\end{eqnarray*}
it results (\ref{el5}).
\end{proof}
\begin{proposition}\label{p1}
Let $(M,g(t),\phi(t)),\,\,\,t\in[0,T)$ be a solution of  the  flow (\ref{rb0}) on the smooth  closed oriented Riemannain manifold $(M^{n}, g_{0},\phi_{0})$  for $\rho<\frac{1}{2(n-1)}$. If $\lambda(t)$ denotes the evolution the first non-zero eigenvalue of the weighted $p$-Laplacian $\Delta_{p,\phi}$ corresponding to the eigenfunction $f(t)$  under the  flow (\ref{rb0}), then
\begin{eqnarray}\nonumber
\frac{\partial }{\partial t}\lambda(t,f(t))|_{t=t_{0}}&=&\lambda(t_{0})(1-n\rho)\int_{M} R|f|^{p}\,d\mu-(1+\rho p-\rho n)\int_{M}R|\nabla f|^{p}d\mu
\\\label{e2}
&&+p\int_{M}Z R^{ij}\nabla_{i} f\nabla_{j} f\,d\mu
+\lambda(t_{0})\int_{M}(\Delta \phi) |f|^{p}\,d\mu\\\nonumber&&-\int_{M}(\Delta \phi)|\nabla f|^{p}d\mu.
\end{eqnarray}
\end{proposition}
\begin{proof}
Let $f(t)$ be a smooth function where $f(t_{0})$ is the corresponding eigenfunction to $\lambda(t_{0})=\lambda(t_{0},f(t_{0}))$. $\lambda(t,f(t))$ is a smooth function and taking derivative of both sides  $\lambda(t,f(t))=-\int_{M}f\Delta_{p,\phi}f\,d\mu$ with respect to time, we get
\begin{equation}\label{e3}
\frac{\partial }{\partial t}\lambda(t,f(t))|_{t=t_{0}}=-\frac{\partial }{\partial t}\int_{M}f\Delta_{p,\phi}f\,d\mu.
\end{equation}
Now by applying  condition $\int_{M}|f|^{p}d\mu=1$ and the time derivative, we can have
\begin{equation}\label{e4}
\frac{\partial }{\partial t}\int_{M}|f|^{p}d\mu=0=\frac{\partial }{\partial t}\int_{M}|f|^{p-2}f^{2} d\mu=\int_{M}(p-1)|f|^{p-2}f f_{t}d\mu+\int_{M}|f|^{p-2}f\frac{\partial }{\partial t}(f d\mu),
\end{equation}
hence
\begin{equation}\label{e5}
\int_{M}|f|^{p-2}f\left[ (p-1) f_{t}d\mu+\frac{\partial }{\partial t}(f d\mu)\right]=0.
\end{equation}
On the other hand, using (\ref{el5}), we obtain
\begin{eqnarray}\nonumber
\frac{\partial }{\partial t}\int_{M}f\Delta_{p,\phi}f\,d\mu&=&\int_{M}\frac{\partial }{\partial t}(\Delta_{p,\phi}f)f\,d\mu+\int_{M}\Delta_{p,\phi}f\,\frac{\partial }{\partial t}(f\,d\mu)\\\nonumber
&=&2\int_{M}R^{ij}\nabla_{i}(Z\nabla_{j}f)f\,d\mu-2\rho \int_{M}R\Delta_{p,\phi}f f\,d\mu\\\label{eq2}
&&+\int_{M}g^{ij}\nabla_{i}(Z_{t}\nabla_{j}f)f\,d\mu
+\int_{M}g^{ij}\nabla_{i}(Z\nabla_{j}f_{t})f\,d\mu\\\nonumber&&+\rho(n-2) \int_{M}Z\nabla R. \nabla f f\,d\mu-\int_{M}Z_{t}\nabla\phi. \nabla f f\,d\mu\\\nonumber&&-\int_{M}Z\nabla\phi_{t}. \nabla ff\,d\mu
-\int_{M}Z\nabla\phi. \nabla f_{t}f\,d\mu\\\nonumber&&-2\int_{M}R^{ij}Z\nabla_{i}\phi \nabla_{j}ff\,d\mu
-\int_{M}\lambda|f|^{p-2}f\frac{\partial }{\partial t}(f d\mu).
\end{eqnarray}
By the application of integration by parts we can conclude that
\begin{equation}\label{e6}
\int_{M}g^{ij}\nabla_{i}(Z_{t}\nabla_{j}f)f\,d\mu=-\int_{M}Z_{t}|\nabla f|^{2}d\mu+\int_{M}Z_{t}\nabla f.\nabla\phi f\,d\mu,
\end{equation}
similarly integration by parts implies that
\begin{equation}\label{e7}
\int_{M}g^{ij}\nabla_{i}(Z\nabla_{j}f_{t})f\,d\mu=-\int_{M}Z\nabla f_{t}.\nabla f\,d\mu+\int_{M}Z\nabla f_{t}.\nabla\phi f\,d\mu,
\end{equation}
and
\begin{eqnarray}\nonumber
\int_{M}R^{ij}\nabla_{i}(Z\nabla_{j}f)f\,d\mu&=&-\int_{M}ZR^{ij}\nabla_{i} f\nabla_{j} f\,d\mu+\int_{M}ZR^{ij}\nabla_{j} f\nabla_{i}\phi f\,d\mu\\\label{e8}&&
-\int_{M}Z\nabla_{i}R^{ij}\nabla_{j}f f\,d\mu.
\end{eqnarray}
But we can write
\begin{eqnarray}\nonumber
2\int_{M}Z\nabla_{i}R^{ij}\nabla_{j}f f\,d\mu&=&2\int_{M}Zg^{ik}g^{jl}\nabla_{j}f\nabla_{i}R_{kl}f\,d\mu=\int_{M}Zg^{jl}\nabla_{j}f\nabla_{l}R f\,d\mu\\\label{e9}
&=&-\int_{M}R\Delta_{p,\phi}f\,f\,d\mu-\int_{M}R|\nabla f|^{p}d\mu.
\end{eqnarray}
Putting  (\ref{e9}) in (\ref{e8}), yields
\begin{eqnarray}\nonumber
2\int_{M}R^{ij}\nabla_{i}(Z\nabla_{j}f) f\,d\mu&=&-2\int_{M}ZR^{ij}\nabla_{i} f\nabla_{j} f\,d\mu+2\int_{M}ZR^{ij}\nabla_{j} f\nabla_{i}\phi f\,d\mu\\\label{e12}
&&-\int_{M}\lambda R|f|^{p}\,d\mu+\int_{M}R|\nabla f|^{p}d\mu.
\end{eqnarray}
Now, replacing (\ref{e6}), (\ref{e7}) and (\ref{e12}) in (\ref{eq2}), we obtain
\begin{eqnarray}\nonumber
\frac{\partial }{\partial t}\int_{M}f\Delta_{p,\phi}f\,d\mu&=&-2\int_{M}ZR^{ij}\nabla_{i} f\nabla_{j} f\,d\mu
-\int_{M}\lambda R|f|^{p}\,d\mu+\int_{M}R|\nabla f|^{p}d\mu\\\nonumber
&&+2\rho \int_{M}\lambda R|f|^{p}d\mu+\rho (n-2)\int_{M}Z\nabla R.\nabla f fd\mu\\\label{e13}
&&-\int_{M}Z_{t}|\nabla f|^{2}d\mu-\int_{M}Z\nabla f_{t}.\nabla f\,d\mu-\int_{M}Z\nabla \phi_{t}.\nabla f\,f\,d\mu\\\nonumber
&&-\int_{M}\lambda|f|^{p-2}f\frac{\partial }{\partial t}(f d\mu).
\end{eqnarray}
On the other hand of Lemma \ref{l1} we have
\begin{equation}\label{e14}
Z_{t}=\frac{\partial }{\partial t}(|\nabla f|^{p-2})=(p-2)|\nabla f|^{p-4}\{R^{ij}\nabla_{i}f\nabla_{j}f-\rho R |\nabla f|^{2}+g^{ij}\nabla_{i}f\nabla_{j}f_{t} \}.
\end{equation}
Therefore putting this into (\ref{e13}), we get
\begin{eqnarray}\nonumber
-\frac{\partial }{\partial t}\lambda(t,f(t))|_{t=t_{0}}&=&-p\int_{M}ZR^{ij}\nabla_{i} f\nabla_{j} f\,d\mu
+\lambda(t_{0})(2\rho-1)\int_{M} R|f|^{p}\,d\mu\\\nonumber
&&+(1+\rho p-2\rho)\int_{M}R|\nabla f|^{p}d\mu+\rho (n-2)\int_{M}Z\nabla R.\nabla f fd\mu
\\\nonumber
&&
-(p-1)\int_{M}Z\nabla f_{t}.\nabla f\,d\mu-\int_{M}Z\nabla \phi_{t}.\nabla f\,f\,d\mu\\\label{e15}
&&-\lambda(t_{0})\int_{M}|f|^{p-2}f\frac{\partial }{\partial t}(f d\mu).
\end{eqnarray}
Also
\begin{eqnarray}\nonumber
-(p-1)\int_{M}Z\nabla f_{t}.\nabla f\,d\mu&=&(p-1)\int_{M}\nabla(Z\nabla f) f_{t}\,d\mu-(p-1)\int_{M}Z\nabla f.\nabla\phi f_{t}\,d\mu\\\label{e16}
&=&(p-1)\int_{M}f_{t}\Delta_{p,\phi}f\,d\mu=- (p-1)\int_{M}\lambda|f|^{p-2}f\,f_{t}\,d\mu.
\end{eqnarray}
Then we arrive at
\begin{eqnarray}\nonumber
-\frac{\partial }{\partial t}\lambda(t,f(t))|_{t=t_{0}}&=&-p\int_{M}ZR^{ij}\nabla_{i} f\nabla_{j} f\,d\mu
+\lambda(t_{0})(2\rho-1)\int_{M} R|f|^{p}\,d\mu\\\nonumber
&&+(1+\rho p-2\rho )\int_{M}R|\nabla f|^{p}d\mu+\rho (n-2)\int_{M}Z\nabla R.\nabla f fd\mu\\\label{e17}
&&-\int_{M}Z\nabla \phi_{t}.\nabla f\,fd\mu\\\nonumber
&&-\lambda(t_{0})\int_{M}|f|^{p-2}f\left((p-1)f_{t}\,d\mu+\frac{\partial }{\partial t}(f d\mu)\right).
\end{eqnarray}
Hence (\ref{e5}) results that
\begin{eqnarray}\nonumber
-\frac{\partial }{\partial t}\lambda(t,f(t))|_{t=t_{0}}&=&-p\int_{M}ZR^{ij}\nabla_{i} f\nabla_{j} f\,d\mu
+\lambda(t_{0})(2\rho-1)\int_{M} R|f|^{p}\,d\mu\\\label{e18}
&&+(1+\rho p-2\rho )\int_{M}R|\nabla f|^{p}d\mu+\rho (n-2)\int_{M}Z\nabla R.\nabla f fd\mu\\\nonumber
&&-\int_{M}Z\nabla \phi_{t}.\nabla f\,fd\mu.
\end{eqnarray}
By integration by parts we get
\begin{equation}\label{e19}
\int_{M}Z\nabla \phi_{t}.\nabla f\,f\,d\mu=\int_{M}\lambda |f|^{p}(\Delta \phi)\,d\mu-\int_{M}(\Delta \phi)|\nabla f|^{p}d\mu
\end{equation}
and
\begin{equation}\label{e191}
\int_{M}Z\nabla R.\nabla f\,f\,d\mu=\int_{M}\lambda R|f|^{p}\,d\mu-\int_{M}R|\nabla f|^{p}d\mu.
\end{equation}
Plugin (\ref{e19}) and (\ref{e191}) into (\ref{e18})  imply  that (\ref{e2}).
\end{proof}
\begin{corollary}\label{c3}
Let $(M,g(t)),\,\,\,t\in[0,T)$ be a solution of  the Ricci-Bourguignon flow (\ref{rb}) on the smooth  closed oriented Riemannain manifold $(M^{n}, g_{0})$ for $\rho<\frac{1}{2(n-1)}$. If $\lambda(t)$ denotes the evolution the first non-zero eigenvalue of the weighted $p$-Laplacian $\Delta_{p,\phi}$ corresponding to the eigenfunction $f(x,t)$  under the Ricci-Bourguignon flow where $\phi$ is independent of $t$, then
\begin{eqnarray}\nonumber
\frac{\partial }{\partial t}\lambda(t,f(t))|_{t=t_{0}}&=&\lambda(t_{0})(1-n\rho)\int_{M} R|f|^{p}\,d\mu-(1+\rho p-\rho n)\int_{M} R|\nabla f|^{p}d\mu
\\\label{e21}&&+p\int_{M}ZR^{ij}\nabla_{i} f\nabla_{j} f\,d\mu.
\end{eqnarray}
\end{corollary}
We can get the evolution for the first eigenvalue of the geometric operator $\Delta_{p}$ under the Ricci-Bourguignon flow (\ref{rb}) and along the Ricci flow, which studied in \cite{JY}. Also, in Corollary \ref{c3}, if $p=2$ then we can obtain the evolution for the first eigenvalue of the Witten-Laplace operator along the the Ricci-Bourguignon flow (\ref{rb}), which investigate  in \cite{Az2}.
\begin{theorem}\label{t1}
Let $(M,g(t),\phi(t)),\,\,\,t\in[0,T)$ be a solution of  the     flow (\ref{rb0}) on the smooth  closed oriented Riemannain manifold $(M^{n}, g_{0})$ for $\rho<\frac{1}{2(n-1)}$. Let $R_{ij}-(\beta R+\gamma\Delta \phi) g_{ij}\geq0$, $\beta\geq \frac{1+\rho(p-n)}{p}$ and $\gamma\geq \frac{1}{p}$  along the  flow (\ref{rb0}) and  $R<\Delta \phi$ in $M\times[0,T)$. Suppose that  $\lambda(t)$ denotes the evolution the first non-zero eigenvalue of the weighted $p$-Laplacian $\Delta_{p,\phi}$ then
\begin{enumerate}
  \item  If $R_{\min}(0)\geq 0$, then $\lambda(t)$ is nondecreasing  along the  Ricci-Bourguignon flow for any $t\in[0,T)$.
  \item  If $R_{\min}(0)> 0$, then the quantity $\lambda(t)(n-2R_{\min}(0)t)^{\frac{1}{n}}$ is nondecreasing  along the  Ricci-Bourguignon flow for  $T\leq \frac{n}{2R_{\min}(0)}$.
  \item If $R_{\min}(0)<0$, then the quantity $\lambda(t)(n-2R_{\min}(0)t)^{\frac{1}{n}}$ is nondecreasing  along the  Ricci-Bourguignon flow for any $t\in[0,T)$.
\end{enumerate}

\end{theorem}
\begin{proof}
According to (\ref{e2}) of Proposition \ref{p1}, we have
\begin{eqnarray}\nonumber
\frac{\partial }{\partial t}\lambda(t,f(t))|_{t=t_{0}}&\geq&\lambda(t_{0})(1-n\rho)\int_{M} R|f|^{p}\,d\mu-(1+\rho p-\rho n)\int_{M}R|\nabla f|^{p}d\mu
\\\label{e2p}
&&+p\beta \int_{M}R|\nabla f|^{p}\,d\mu+p\gamma\int_{M}(\Delta \phi)|\nabla f|^{p}d\mu
\\\nonumber&&+\lambda(t_{0})\int_{M}R|f|^{p}\,d\mu-\int_{M}(\Delta \phi)|\nabla f|^{p}d\mu
\\\nonumber&=&\lambda(t_{0})(2-n\rho)\int_{M} R|f|^{p}\,d\mu+(p\gamma-1)\int_{M}R|\nabla f|^{p}d\mu
\\\nonumber&&+[p\beta-(1+\rho p-\rho n)]\int_{M}R|\nabla f|^{p}d\mu.
\end{eqnarray}
On  the other hand, the scalar curvature along the  Ricci-Bourguignon  flow evolves by
\begin{equation}\label{va}
\frac{\partial R}{\partial t}=(1-2(n-1)\rho)\Delta R+2|Ric|^{2}-2\rho R^{2}.
\end{equation}
The  inequality $|Ric|^{2}\geq \frac{R^{2}}{n}$ yields
\begin{equation}\label{tmp}
\frac{\partial R}{\partial t}\geq (1-2(n-1)\rho)\Delta R+2(\frac{1}{n}-\rho) R^{2}.
\end{equation}
Since the solution to the corresponding ODE $y'=2(\frac{1}{n}-\rho)y^{2}$ with initial value $c=\mathop {\min}\limits_{x \in M}\,R(0)=R_{\min}(0)$ is
\begin{equation}\label{v1a}
\sigma(t)=\frac{nc}{n-2(1-n\rho)ct}.
\end{equation}
Notic that $\sigma(t)$ defined on $[0,T')$ where $T'=\min\{T, \frac{n}{2(1-n)\rho c}\}$ when $c>0$ and on $[0,T)$ when  $c\leq 0$.
Using the maximum principle to (\ref{tmp}), we have $R_{g(t)}\geq \sigma(t) $. Therefore (\ref{e2p}) becomes
\begin{equation*}
\frac{d}{dt}\lambda(t,f(t))|_{t=t_{0}}\geq A\lambda(t_{0}) \sigma(t_{0}),
\end{equation*}
where
$ A=p(\beta+\gamma)-\rho (p+2n)$ and  this results that in any sufficiently small neighborhood of $t_{0}$ as $I_{0}$,  we obtain
$$\frac{d}{dt}\lambda(t,f(t))\geq A\lambda(f,t) \sigma(t). $$
Integrating of both sides of  the last inequality with respect to $t$ on $[t_{1}, t_{0}]\subset I_{0}$, we have
\begin{equation*}
\ln \frac{\lambda(t_{0},f(t_{0}))}{ \lambda(f(t_{1}),t_{1})}>\ln (\frac{n-2(1-n\rho)ct_{1}}{n-2(1-n\rho)ct_{0}})^{\frac{nA}{2(1-n\rho)}}.
\end{equation*}
Since $\lambda(t_{0},f(t_{0}))=\lambda(t_{0})$ and  $\lambda(f(t_{1}),t_{1})\geq\lambda(t_{1})$ we conclude that
\begin{equation*}
\ln \frac{\lambda(t_{0})}{ \lambda(t_{1})}>\ln (\frac{n-2(1-n\rho)ct_{1}}{n-2(1-n\rho)ct_{0}})^{\frac{nA}{2(1-n\rho)}},
\end{equation*}
 that is the quantity  $\lambda(t)(n-2(1-n\rho)ct)^{\frac{nA}{2(1-n\rho)}}$ is strictly increasing  in any sufficiently small neighborhood of $t_{0}$. Since $t_{0}$ is arbitrary, then $\lambda(t)(n-2(1-n\rho)ct)^{\frac{nA}{2(1-n\rho)}}$ is strictly increasing along the   flow  (\ref{rb0}) on $[0,T)$. Now we have,
\begin{enumerate}
  \item If $R_{\min}(0)\geq0$, by the non-negatively  of $R_{g(t)}$  preserved along the Ricci-Bourguignon flow hence $\frac{d}{dt}\lambda(t,f(t))\geq0$, consequently $\lambda(t)$ is strictly increasing along the   flow (\ref{rb}) on $[0,T)$.
\item  If $R_{\min}(0)>0$ then $\sigma(t)$ defined on $[0,T')$, thus the quantity  $\lambda(t)(n-2(1-n\rho)ct)^{\frac{nA}{2(1-n\rho)}}$ is nondecreasing along the flow (\ref{rb}) on $[0,T')$.
\item If $R_{\min}(0)<0$ then $\sigma(t)$ defined on $[0,T')$, thus the quantity  $\lambda(t)(n-2(1-n\rho)ct)^{\frac{nA}{2(1-n\rho)}}$ is nondecreasing along the flow (\ref{rb}) on $[0,T')$.
\end{enumerate}
\end{proof}
\begin{theorem}
Let $(M^{n},g(t),\phi(t))$, $t\in[0,T)$ be a solution of the flow (\ref{rb0}) on a closed Riemannian manifold $(M^{n},g_{0})$ with $R(0)>0$ for $\rho<\frac{1}{2(n-1)}$. Let $\lambda(t)$ be the first eigenvalue of the weighted  $p$-Laplacian $\Delta_{p,\phi}$, then $\lambda(t)\to+\infty$ in finite time for $p\geq2$  where $Ric-\nabla\phi\otimes\nabla\phi\geq \beta R g$ in $M\times[0,T)$ and $\beta\in[0,\frac{1}{n}]$ is a constant.
\end{theorem}
\begin{proof}
The weighted $p$-Reilly formula on closed Riemannian manifolds (see \cite{WL}) as follows
\begin{equation}\label{e36}
\int_{M}\left[(\Delta_{p,\phi}f)^{2}-|\nabla f|^{2p-4}|Hess \,f|_{A}^{2}\right]
d\mu=\int_{M}|\nabla f|^{2p-4}(Ric+\nabla^{2}\phi)(\nabla f,\nabla f)\,d\mu,
\end{equation}
where $f\in C^{\infty}(M)$ and
\begin{equation}\label{e37}
|Hess\,f|_{A}^{2}=|Hess\,f|^{2}+\frac{p-2}{2}\frac{|\nabla|\nabla f|^{2} f|^{2}}{|\nabla f|^{2}}+\frac{(p-2)^{2}}{4}\frac{<\nabla f,\nabla |\nabla f|^{2}>^{2}}{|\nabla f|^{4}}.
\end{equation}
By a straightforward computation we have the following inequality
\begin{eqnarray}\nonumber
|\nabla f|^{2p-4}|Hess \,f|_{A}^{2}&\geq&\frac{1}{n}\left( \Delta_{p,\phi}f+|\nabla f|^{p-2}<\nabla\phi,\nabla f>\right)^{2}\\\label{e39}
&\geq& \frac{1}{1+n}(\Delta_{p,\phi}f)^{2}-|\nabla f|^{2p-4}|\nabla\phi.\nabla f|^{2}.
\end{eqnarray}
Recall that $\Delta_{p,\phi}f=-\lambda |f|^{p-2}f$, which implies
\begin{equation}\label{e40}
\int_{M}(\Delta_{p,\phi}f)^{2} d\mu=\lambda^{2}\int_{M} |f|^{2p-2}d\mu.
\end{equation}
Combining (\ref{e39}) and (\ref{e40}) we can write
\begin{eqnarray}\nonumber
\int_{M}\left[(\Delta_{p,\phi}f)^{2}-|\nabla f|^{2p-4}|Hess \,f|_{A}^{2}\right]
d\mu&\leq&(1-  \frac{1}{1+n})\lambda^{2}\int_{M} |f|^{2p-2}d\mu\\
\label{e41}&&+\int_{M}|\nabla f|^{2p-4}|\nabla\phi.\nabla f|^{2}d\mu,
\end{eqnarray}
putting  (\ref{e41}) in (\ref{e36})  yields
\begin{eqnarray}\nonumber
(1-  \frac{1}{1+n})\lambda^{2}\int_{M} |f|^{2p-2}d\mu&+&\int_{M}|\nabla f|^{2p-4}|\nabla\phi.\nabla f|^{2}d\mu\geq
\int_{M}|\nabla f|^{2p-4}Ric(\nabla f,\nabla f)\,d\mu\\\label{e42}
&&+\int_{M}|\nabla f|^{2p-4}\nabla^{2}\phi(\nabla f,\nabla f)\,d\mu.
\end{eqnarray}
By identifying $ \nabla\phi\otimes \nabla \phi(\nabla f, \nabla f)$ with $|\nabla\phi.\nabla f|^{2}$ (see \cite{DL}), we obtain
\begin{equation}\label{e43}
\int_{M}|\nabla f|^{2p-4}\ \nabla\phi\otimes \nabla \phi(\nabla f,\nabla f)\,d\mu=\int_{M}|\nabla f|^{2p-4}|\nabla\phi.\nabla f|^{2}d\mu,
\end{equation}
therefore it and  $Ric-\nabla\phi\otimes\nabla\phi\geq \beta R g$ result that
\begin{equation}\label{e44}
(1-  \frac{1}{1+n})\lambda^{2}\int_{M} |f|^{2p-2}d\mu\geq\beta \int_{M}R|\nabla f|^{2p-2}d\mu+\int_{M}|\nabla f|^{2p-4}\nabla^{2}\phi(\nabla f,\nabla f)\,d\mu.
\end{equation}
Now, since $\phi$ satisfies in $\phi_{t}=\Delta\phi$, we get
\begin{equation}\label{e45}
|\nabla^{2}\phi|\geq \frac{1}{\sqrt{n}}|\Delta \phi|=\frac{1}{\sqrt{n}}| \phi_{t}|,
\end{equation}
hence
\begin{eqnarray}\nonumber
(1-  \frac{1}{1+n})\lambda^{2}\int_{M} |f|^{2p-2}d\mu&\geq&\beta \int_{M}R|\nabla f|^{2p-2}d\mu+\frac{1}{\sqrt{n}}\int_{M}| \phi_{t}||\nabla f|^{2p-2}d\mu\\\label{e46}
&\geq& (\beta R_{\min}(t)+\frac{1}{\sqrt{n}}\mathop {\min }\limits_{x \in M}| \phi_{t}|)\int_{M}|\nabla f|^{2p-2}d\mu.
\end{eqnarray}
Multiplying  $\Delta_{p,\phi}f=-\lambda |f|^{p-2}f$ by $ |f|^{p-2}f$ on both sides, we obtain $ |f|^{p-2}f\Delta_{p,\phi}f=-\lambda |f|^{2p-2}f$. Then integrating by parts  and the H\"{o}lder inequality for $p>2$, we obtain
 \begin{eqnarray*}
\lambda\int_{M}|\nabla f|^{2p-2}d\mu&=&-\int_{M}|f|^{p-2}f\Delta_{p,\phi}f\,d\mu=(p-1)\int_{M}|\nabla f|^{p}|f|^{p-2}d\mu\\
&\leq&(p-1)\left(\int_{M}(|\nabla f|^{p})^{\frac{2p-2}{p}}d\mu\right)^{\frac{p}{2p-2}}\left(\int_{M}(|f|^{p-2})^{\frac{2p-2}{p-2}}d\mu\right)^{\frac{p-2}{2p-2}}\\
&=&(p-1)\left(\int_{M}|\nabla f|^{2p-2}d\mu\right)^{\frac{p}{2p-2}}\left(\int_{M}|f|^{2p-2}d\mu\right)^{\frac{p-2}{2p-2}},
\end{eqnarray*}
so we can conclude that
\begin{equation*}
\int_{M}|\nabla f|^{2p-2}d\mu\geq (\frac{\lambda}{p-1})^{\frac{2p-2}{p}}\int_{M}|f|^{2p-2}d\mu
\end{equation*}
and it implies
\begin{equation*}
(1-  \frac{1}{1+n})\lambda^{2}\int_{M} |f|^{2p-2}d\mu\geq (\beta R_{\min}(t)+\frac{1}{\sqrt{n}}\mathop {\min }\limits_{x \in M}| \phi_{t}|)(\frac{\lambda}{p-1})^{\frac{2p-2}{p}}\int_{M}|f|^{2p-2}d\mu,
\end{equation*}
more precisely
\begin{equation*}
\left[ (1-  \frac{1}{1+n})\lambda^{2}- (\beta R_{\min}(t)+\frac{1}{\sqrt{n}}\mathop {\min }\limits_{x \in M}| \phi_{t}|)(\frac{\lambda}{p-1})^{\frac{2p-2}{p}}\right]\int_{M}|f|^{2p-2}d\mu\geq0.
\end{equation*}
Since $\int_{M}|f|^{2p-2}d\mu\geq0$, for $p>2$ we get
\begin{equation*}
\lambda(t)\geq \left[  (\beta R_{\min}(t)+\frac{1}{\sqrt{n}}\mathop {\min }\limits_{x \in M}| \phi_{t}|) \frac{1+n\alpha}{1+n\alpha-\alpha}\right]^{\frac{p}{2}}\frac{1}{(p-1)^{(p-1)}}.
\end{equation*}
Since $R_{\min}(t)\to +\infty$ (see \cite{GC}) and $\mathop {\min }\limits_{x \in M}| \phi_{t}|$ is finite then $\lambda(t)\to +\infty$. For $p=2$, (\ref{e46}) results that
\begin{equation*}
(1-  \frac{1}{1+n})\lambda^{2}\int_{M} |f|^{2}d\mu\geq (\beta R_{\min}(t)+\frac{1}{\sqrt{n}}\mathop {\min }\limits_{x \in M}| \phi_{t}|)\lambda\int_{M}| f|^{2}d\mu,
\end{equation*}
hence
\begin{equation*}
\lambda(t)\geq (\beta R_{\min}(t)+\frac{1}{\sqrt{n}}\mathop {\min }\limits_{x \in M}| \phi_{t}|) \frac{1+n\alpha}{1+n\alpha-\alpha}.
\end{equation*}
This implies that $\lambda(t)\to +\infty$.
\end{proof}
\begin{corollary}
Let $(M, g(t))$, $t\in[0,T)$  be a solution of the Ricci-Bourguignon flow (\ref{rb}) on  the smooth  closed Riemannnian manifold $(M^{3}, g_{0})$, $\phi$ is independent of $t$,   $\frac{1}{6}<\rho <\frac{1}{4}$ and $\lambda(t)$ be the first eigenvalue of the weighted  $p$-Laplacian $\Delta_{p,\phi}$. If $R_{ij}>\frac{1+\rho p-3\rho}{p}Rg_{ij}$ on $M^{n}\times\{0\}$ and $c=R_{\min}(0)\geq0$
then the quantity $\lambda(t)(3-2(1-3\rho)ct)^{\frac{3}{2}}$ is nondecreasing along the flow (\ref{rb}) for $p\geq 3$.
\end{corollary}
\begin{proof}
The pinching inequality $R_{ij}>\frac{1+\rho p-3\rho}{p}Rg_{ij}$ for $\frac{1}{6}<\rho< \frac{1}{4}$ and  $p\geq 3$  is preserved along the Ricci-Bourguignon  flow, therefore we have,
\begin{equation*}
R_{ij}>\frac{1+\rho p-3\rho}{p}Rg_{ij},\,\,\,\,\,\,\,\text{on}\,\,[0,T)\times M.
\end{equation*}
Now according to Corollary \ref{c3} we get
\begin{equation*}
\frac{\partial }{\partial t}\lambda(t,f(t))|_{t=t_{0}}\geq\lambda(t_{0})(1-n\rho)\int_{M} R|f|^{p}\,d\mu
\end{equation*}
hence similar to proof of Theorem \ref{t1} we have $R_{g(t)}\geq \sigma(t)$ on $[0,T)$ and then
\begin{equation*}
\frac{\partial }{\partial t}\lambda(t,f(t))|_{t=t_{0}}\geq\lambda(t_{0})(1-n\rho)\sigma(t_{0})
\end{equation*}
thus we arrive at the
the quantity $\lambda(t)(3-2(1-3\rho)ct)^{\frac{3}{2}}$ is nondecreasing.
\end{proof}
\begin{theorem}\label{tt1}
Let $(M,g(t),\phi(t)),\,\,\,t\in[0,T)$ be a solution of  the     flow (\ref{rb0}) on the smooth  closed oriented Riemannain manifold $(M^{n}, g_{0})$ for $\rho<\frac{1}{2(n-1)}$. Let $0<R_{ij}<\frac{1+p\rho-n\rho}{p}Rg_{ij}$ on $M^{n}\times[0,T)$ and  $R<\Delta \phi$ in $M\times[0,T)$. Suppose that  $\lambda(t)$ denotes the evolution the first non-zero eigenvalue of the weighted $p$-Laplacian $\Delta_{p,\phi}$  and  $C=R_{\max}(0)$
then the quantity $\lambda(t)(1-CAt)^{\frac{n\rho-1}{A}}$ is strictly decreasing along  the    flow (\ref{rb0})  on $[0,T')$ where $T'=\min\{T, \frac{1}{CA}\}$ and $A=2\big(n(\frac{1-(n-p)\rho}{p})^{2}-\rho\big)$.
\end{theorem}
\begin{proof}
The proof is similar to proof  of Theorem \ref{t1} with the difference that we need to estimate the upper bound of the right hand  (\ref{e2}).
Notice that $ R_{ij}<\frac{1+p\rho-n\rho}{p}Rg_{ij}$ implies that $|Ric|^{2}<n(\frac{1+p\rho-n\rho}{p})^{2}R^{2}$.
So the evolution of the scalar curvature under the  Ricci-Bourguignon  flow evolve by (\ref{va}) and it  yields
\begin{equation}\label{tmp1}
\frac{\partial R}{\partial t}\leq (1-2(n-1)\rho)\Delta R+2\big(n(\frac{1+p\rho-n\rho}{p})^{2}-\rho\big) R^{2}.
\end{equation}
Applying the maximum principle to (\ref{tmp1}) we have $0\leq R_{g(t)}\leq \gamma(t) $ where
\begin{equation*}
\gamma(t)=\left[C^{-1}- 2\big(n(\frac{1+p\rho-n\rho}{p})^{2}-\rho\big)t  \right]^{-1}=\frac{C}{1-CAt}\,\,\,\,\,\,\text{on}\,\,\,[0,T').
\end{equation*}
Replacing $0\leq R_{g(t)}\leq \gamma(t) $ and $ R_{ij}<\frac{1-(n-2)\rho}{2}Rg_{ij}$ into  equation (\ref{e2}) we can write
 $\frac{d}{dt}\lambda(t,f(t))\leq \frac{(1-n\rho)C}{1-CAt} \lambda(t,f(t))$  in any sufficiently small neighborhood of $t_{0}$, hence with a sequence of calculation the quantity $\lambda(t)(1-CAt)^{\frac{n\rho-1}{A}}$ is strictly decreasing.
\end{proof}

\begin{theorem}\label{tt2}
Let $(M, g(t))$, $t\in[0,T)$ be a solution of the Ricci-Bourguignon  flow (\ref{rb}) on a closed manifold  $M^{n}$  and $\rho<\frac{1}{2(n-1)}$. Let  $\lambda(t)$ be the first nonzero eigenvalue of  the weighted $p$-Laplacian  of the metric $g(t)$ and $\phi$ be independent of $t$. If there is a non-negative constant $a$ such that
\begin{equation}\label{cc1}
 R_{ij}-\frac{1-(n-p)\rho}{p}Rg_{ij}\geq -ag_{ij}\,\,\,\,\,\,\text{in}\,\,\,M^{n}\times[0,T)
\end{equation}
and
\begin{equation}\label{cc2}
R\geq \frac{pa}{1-n\rho}\,\,\,\,\,\,\text{in}\,\,\,M^{n}\times\{0\}
\end{equation}
then $\lambda(t)$ is strictly monotone increasing  along the  Ricci-Bourguignon  flow.
\end{theorem}
\begin{proof}
By Corollary  \ref{c3}, we write evolution of first eigenvalue as follows
\begin{eqnarray}\nonumber
\frac{d}{dt}\lambda(t,f(t))|_{t=t_{0}}&=&(1-n\rho)\lambda(t_{0})\int_{M}R\,f^{2}d\mu\\\label{cc3}&&
+p\int_{M}(R_{ij}-\frac{1-(n-p)\rho}{p}Rg_{ij})|\nabla f|^{p-2}\nabla_{i}f\nabla_{j}f\,d\mu\\\nonumber
&\geq&(1-n\rho)\lambda(t_{0})\int_{M}R\,f^{2}d\mu-ap\int_{M}|\nabla f|^{p}d\mu\geq 0
\end{eqnarray}
combining (\ref{cc1}), (\ref{cc2}) and (\ref{cc3}), we arrive at $\frac{d}{dt}\lambda(f(t),t)>0$  in any sufficiently small neighborhood of $t_{0}$. Since $t_{0}$ is arbitrary, then $\lambda(t)$ is strictly increasing along the Ricci-Bourguignon  flow on $[0,T)$.
\end{proof}
\subsection{Variation of $\lambda(t)$ on a surface}
Now, we rewrite Proposition \ref{p1}  and Corollary \ref{c3} in some remarkable particular cases.
\begin{corollary}
Let  $(M^{2}, g(t))$, $t\in[0,T)$  be  a solution of  the  Ricci-Bourguignon  flow
on a closed Riemannnian surface $(M^{2}, g_{0})$ for $\rho<\frac{1}{2}$. If $\lambda(t)$ denotes the
evolution of the first eigenvalue of the weighted $p$-Laplacian under the Ricci-Bourguignon flow, then
\begin{enumerate}
  \item If $\frac{\partial\phi}{\partial t}=\Delta\phi$ then
\begin{eqnarray}\nonumber
\frac{d}{dt}\lambda(t,f(t))|_{t=t_{0}}&=&(1-2\rho)\lambda(t_{0})\int_{M}R\,|f|^{p}d\mu-(1+\rho\phi-2\rho-\frac{p}{2})\int_{M}R|\nabla f|^{p}d\mu\\\label{R21}&&+\lambda(t_{0})\int_{M}(\Delta \phi)|f|^{p}d\mu-\int_{M}(\Delta \phi)|\nabla f|^{p}d\mu.
\end{eqnarray}
  \item If  $\phi$ is  independent of $t$ then
\begin{equation}\label{R221}
\frac{d}{dt}\lambda(t,f(t))|_{t=t_{0}}=(1-2\rho)\lambda(t_{0})\int_{M}R\,|f|^{p}d\mu-(1+\rho\phi-2\rho-\frac{p}{2})\int_{M}|\nabla f|^{p}d\mu.
\end{equation}
\end{enumerate}
\end{corollary}
\begin{proof}
In dimension $n=2$, we have $Ric=\frac{1}{2}Rg$, then (\ref{e2}) and (\ref{e21}) imply  that (\ref{R21}) and (\ref{R221}) respectively.
\end{proof}
\begin{lemma}
Let  $(M^{2}, g(t))$, $t\in[0,T)$  be  a solution of  the  Ricci-Bourguignon  flow
on a closed surface $(M^{2}, g_{0})$ with nonnegative scalar curvature for $\rho<\frac{1}{2}$, $\phi$ be  independent of $t$ and $p\geq 2$. If $\lambda(t)$ denotes the
evolution of  the first  eigenvalue of the weighted $p$-Laplacian  under the Ricci-Bourguignon flow, then
\begin{equation*}
\frac{\lambda(0)}{(1-c(1-2\rho)t)^{\frac{p}{2}}}\leq\lambda(t)
\end{equation*}
on $(0,T')$ where $c=\mathop {\min}\limits_{x \in M}\,R(0)$ and  $T'=\min\{T, \frac{1}{c(1-2\rho)}\}$.
\end{lemma}
\begin{proof}
On a surface  we have $Ric=\frac{1}{2}Rg$, and  for the scalar curvature $R$ on a closed surface $M$ along  the Ricci-Bourguignon flow we get
\begin{equation}
 \frac{c}{1-c(1-2\rho)t}\leq R,\,\,\,\,\,\,\,\,\,\,\, \text{on}\,\,\,\,\,\,[0,T')
\end{equation}
where $T'=\min\{T, \frac{1}{c(1-2\rho)}\}$. According to (\ref{R221})  and $\int_{M}|f|^{p}d\mu=1$ we have
\begin{equation}
 \frac{p}{2}\frac{c(1-2\rho)\lambda(t,f(t))}{1-c(1-2\rho)t}\leq \frac{d}{dt}\lambda(t,f(t)),
\end{equation}
 in any  small enough neighborhood of $t_{0}$. After the integrating above inequality  with respect to time $t$, this becomes
\begin{equation*}
\frac{\lambda(0, f(0))}{(1-c(1-2\rho)t)^{\frac{p}{2}}}\leq\lambda(t_{0}).
\end{equation*}
Now $\lambda(0,f (0))\geq \lambda(0)$ results that $\frac{\lambda( 0)}{(1-c(1-2\rho)t)^{\frac{p}{2}}}\leq\lambda(t_{0})$.
Since $t_{0}$ is arbitrary, then $\frac{\lambda(0)}{(1-c(1-2\rho)t)^{\frac{p}{2}}}\leq\lambda(t)$  on $(0,T')$.
\end{proof}
\begin{lemma}
Let $(M^{2},g_{0})$ be a closed surface with   nonnegative scalar curvature and   $\phi$ be  independent of $t$, then the eigenvalues of  the weighted $p$-Laplacian are increasing under   the Ricc-Bourguignon flow for $\rho<\frac{1}{2}$.
\end{lemma}
\begin{proof}
Along  the  Ricci-Bourguignon  flow on a surface, we have
$$\frac{\partial R}{\partial t}=(1-2\rho)(\Delta R+R^{2})$$
by the scalar maximum principle, the nonnegativity of the
 scalar curvature is preserved along the Ricci-Bourguignon flow (see \cite{GC}). Then (\ref{R221}) implies that
$\frac{d}{dt}\lambda(t,f(t))|_{t=t_{0}}>0$,
 this results that in any sufficiently small neighborhood of $t_{0}$ as $I_{0}$,  we get
$\frac{d}{dt}\lambda(t,f(t))>0$. Hence by integrating on interval  $[t_{1}, t_{0}]\subset I_{0}$, we have $\lambda(t_{1},f(t_{1}))\leq \lambda(t_{0},f(t_{0}))$.
Since $\lambda(t_{0},f(t_{0}))=\lambda(t_{0})$ and  $\lambda(t_{1},f(t_{1}))\geq\lambda(t_{1})$ we conclude that
$\lambda(t_{1})\leq \lambda(t_{0})$.
Therefore the quantity  $\lambda(t)$ is strictly increasing  in any sufficiently small neighborhood of $t_{0}$, but $t_{0}$ is arbitrary, then $\lambda(t)$ is strictly increasing along the Ricci-Bourguignon  flow on $[0,T)$.
\end{proof}


\subsection{Variation of $\lambda(t)$ on homogeneous manifolds }
In this section, we consider the behavior of the first eigenvalue when we evolve an initial   homogeneous metric along the flow (\ref{rb0}).
\begin{proposition}\label{P1}
Let $(M^{n},g(t))$ be a solution of the  Ricci-Bourguignon  flow on
the smooth closed  homogeneous  manifold  $(M^{n},g_{0})$ for $\rho<\frac{1}{2(n-1)}$. Let  $\lambda(t)$
be denote the evaluation of an eigenvalue under the Ricci-Bourguignon  flow, then
\begin{enumerate}
  \item If $\frac{\partial\phi}{\partial t}=\Delta\phi$ then
\begin{eqnarray}\nonumber
\frac{d}{dt}\lambda(t,f(t))|_{t=t_{0}}&=&-\rho p R\lambda(t_{0})
+p\int_{M}Z R^{ij}\nabla_{i} f\nabla_{j} f\,d\mu
+\lambda(t_{0})\int_{M}(\Delta \phi) |f|^{p}\,d\mu\\\label{hm1}&&-\int_{M}(\Delta \phi)|\nabla f|^{p}d\mu.
\end{eqnarray}
  \item If  $\phi$ is  independent of $t$ then
\begin{equation}\label{hm2}
\frac{d}{dt}\lambda(t,f(t))|_{t=t_{0}}=-\rho p R\lambda(t_{0})+p\int_{M}Z R^{ij}\nabla_{i} f\nabla_{j} f\,d\mu.
\end{equation}
\end{enumerate}
\end{proposition}
\begin{proof}
Since the evolving metric remains homogeneous and a homogeneous
manifold has constant scalar curvature. Therefore (\ref{e2}) implies
that
\begin{eqnarray*}
\frac{d}{dt}\lambda(t,f(t))|_{t=t_{0}}&=&(1-n\rho)\lambda(t_{0}) R\int_{M}\,f^{2}d\mu+((n-p)\rho-1)R\int_{M}|\nabla
f|^{2}d\mu \\
&&+p\int_{M}Z R^{ij}\nabla_{i} f\nabla_{j} f\,d\mu
+\lambda(t_{0})\int_{M}(\Delta \phi) |f|^{p}\,d\mu\\&&-\int_{M}(\Delta \phi)|\nabla f|^{p}d\mu.
\end{eqnarray*}
But $\int_{M}\,f^{2}d\mu=1$ and $\int_{M}|\nabla f|^{2}d\mu=1$ therefore  last equation results that (\ref{hm1}) and (\ref{hm2}).
\end{proof}
\subsection{Variation of $\lambda(t)$ on $3$-dimensional  manifolds}
In this section, we consider the behavior of $\lambda(t)$ on $3$-dimensional  manifolds.
 \begin{proposition}
 Let $(M^{3},g(t))$ be a solution of the  Ricci-Bourguignon  flow (\ref{rb}) for $\rho<\frac{1}{4}$ on a closed Riemannian manifold $M^{3}$  whose Ricci curvature is initially positive  and there exists  $0\leq\epsilon\leq \frac{1}{3}$ such that
 $$Ric \geq \epsilon Rg.$$
If  $\phi$ is  independent of $t$ and $\lambda(t)$ denotes the
evolution of the first eigenvalue of the weighted $p$-Laplacian under the Ricci-Bourguignon flow then the quantity $e^{-\int_{0}^{t}A(\tau)d\tau}\lambda(t)$ is nondecreasing along  the  Ricci-Bourguignon  flow (\ref{rb})  for $p\leq3$,
where
 \begin{equation*}
A(t)=\frac{3c(1-3\rho)}{3-2(1-3\rho)c t}+(3\rho+p\epsilon-1-\rho p)\left(-2(1-\rho)t+\frac{1}{C}\right)^{-1},
\end{equation*}
$C= R_{\max }(0)$ and $c=R_{\min }(0)$.
\end{proposition}
\begin{proof}
In \cite{GC} has been shown that the pinching inequality $Ric\geq \epsilon Rg $ and  nonnegative scalar curvature are  preserved along  the  Ricci-Bourguignon  flow (\ref{rb}) on closed manifold $M^{3}$, then using
 (\ref{e21})  we obtain
\begin{eqnarray*}
\frac{d}{dt}\lambda(f,t)|_{t=t_{0}}&\geq&(1-3\rho)\lambda(t_{0})\int_{M}R\,f^{2}d\mu+(3\rho-1-\rho p)\int_{M}R|\nabla
f|^{2}d\mu\\&& +p\epsilon\int_{M}R|\nabla f|^{2}d\mu\\
&=&(1-3\rho)\lambda(t_{0})\int_{M}R\,f^{2}d\mu+(3\rho+p\epsilon-1-\rho p)\int_{M}R|\nabla
f|^{2}d\mu,
\end{eqnarray*}
on the other hand
 the scalar curvature under the  Ricci-Bourguignon flow evolves by (\ref{va}) for $n=3$.
By $|Ric|^{2}\leq R^{2}$ we have
\begin{equation*}
\frac{\partial R}{\partial t}\leq(1-4\rho)\Delta R+2(1-\rho) R^{2}.
\end{equation*}
Let $\gamma(t)$ be the solution to the ODE $y'=2(1-\rho)y^{2}$ with initial value $C=R_{\max}(0)$. By the maximum principle, we have
\begin{equation}\label{s}
R(t)\leq\gamma(t)=\left(-2(1-\rho)t+\frac{1}{C}\right)^{-1}
\end{equation}
on $[0,T')$, where $T'=\min\{T,\frac{1}{2(1-\rho)C}\}$. Also, similar to proof of Theorem \ref{t1}, we have
\begin{equation}\label{g}
R(t)\geq \sigma(t)=\frac{3c}{3-2(1-3\rho)c t}\,\,\,\,\text{on}\,\,\,\,[0,T).
\end{equation}
Hence
\begin{eqnarray*}
\frac{d}{dt}\lambda(t,f(t))|_{t=t_{0}}&\geq&(1-3\rho)\lambda(t_{0})\frac{3c}{3-2(1-3\rho)c t_{0}}\\&&+(\rho-1+2\epsilon)\lambda(t_{0})\left(-2(1-\rho)t_{0}+\frac{1}{C}\right)^{-1}\\&=&\lambda(t_{0})A(t_{0})
\end{eqnarray*}
this results that in any sufficiently small neighborhood of $t_{0}$ as $I_{0}$,  we obtain
$$\frac{d}{dt}\lambda(t,f(t))\geq \lambda(f,t) A(t). $$
Integrating of both sides of  the last inequality with respect to $t$ on $[t_{1}, t_{0}]\subset I_{0}$, we can write
\begin{equation*}
\ln \frac{\lambda(t_{0},f(t_{0}))}{ \lambda(t_{1},f(t_{1}))}>\int_{t_{1}}^{t_{0}}A(\tau)d\tau.
\end{equation*}
Since $\lambda(t_{0},f(t_{0}))=\lambda(t_{0})$ and  $\lambda(t_{1},f(t_{1}))\geq\lambda(t_{1})$ we conclude that
\begin{equation*}
\ln \frac{\lambda(t_{0})}{ \lambda(t_{1})}>\int_{t_{1}}^{t_{0}}A(\tau)d\tau.
\end{equation*}
 that is the quantity  $\lambda(t)e^{-\int_{0}^{t}A(\tau)d\tau}$ is strictly increasing  in any sufficiently small neighborhood of $t_{0}$. Since $t_{0}$ is arbitrary, then  $\lambda(t)e^{-\int_{0}^{t}A(\tau)d\tau}$ is strictly increasing along the Ricci-Bourguignon  flow on $[0,T)$.
\end{proof}
\begin{proposition}
Let $(M^{3},g(t))$ be a solution to the  Ricci-Bourguignon flow for $\rho<0$ on a closed homogeneous  $3$-manifold whose Ricci curvature is initially nonnegative and  $\phi$ be  independent of $t$ then the first  eigenvalues of the weighted $p$-Laplacian is  increasing.
\end{proposition}
\begin{proof}
In dimension three,  the Ricci-Bourguignon flow preseves the nonnegativity of the Ricci curvature  is preserved.
From (\ref{hm2}), its implies that $\lambda(t)$ is increasing.
\end{proof}
\section{Example}
In this section, we consider the initial Riemannian manifold $(M^{n},g_{0})$ is Einstein manifold and then find evolving  first eigenvalue  of the weighted $p$-Laplace operator along the Ricci-Bourguignon flow.
\begin{example}
Let $(M^{n},g_{0})$ be an Einstein manifold i.e. there exists a constant $a $ such that $Ric(g_{0})=ag_{0}$. Assume that
 a solution to the Ricci-Bourguignon flow is of the form
\begin{equation*}
g(t)=u(t)g_{0},\,\,\,\, u(0)=1
\end{equation*}
where $u(t)$ is a positive function. By a straitforward computation, we have
$$\frac{\partial g}{\partial t}=u'(t)g_{0},\,\,\,Ric(g(t))=Ric(g_{0})= ag_{0}=\frac{a}{u(t)}g(t),\,\,
R_{g(t)}=\frac{an}{u(t)},$$
for this to be  a solution of the Ricci-Bourguignon  flow, we require
\begin{equation*}
u'(t)g_{0}=-2Ric(g(t))+2\rho R_{g(t)}g(t)= (-2a+2\rho a n) g_{0}
\end{equation*}
this shows that
 $$u(t)=(-2a+2\rho a n) t+1,$$
 so  $g(t)$ is an Einstein metric.
Using formula (\ref{e21}) for evolution of first eigenvalue  along the Ricci-Bourguignon flow, we obtain the following relation
\begin{eqnarray*}
\frac{d}{dt}\lambda(t,f(t))|_{t=t_{0}}&=&(1-n\rho)\frac{an}{u(t_{0})}\lambda(t_{0})\int_{M}\,|f|^{p}d\mu +2\frac{a}{u(t_{0})}\int_{M}|\nabla f|^{p}d\mu\\&&-((p-n)\rho-1)\frac{an}{u(t_{0})}\int_{M}|\nabla f|^{p}d\mu=\frac{pa(1-n\rho)\lambda(t_{0})}{u(t_{0})},
\end{eqnarray*}
this results that in any sufficiently small neighborhood of $t_{0}$ as $I_{0}$,  we get
\begin{equation*}
\frac{d}{dt}\lambda(t,f(t))=\frac{pa(1-n\rho)\lambda(t,f(t))}{(-2a+2\rho a n) t+1}.
\end{equation*}
Integrating the last inequality with respect to $t$ on $[t_{1}, t_{0}]\subset I_{0}$, we have
\begin{equation*}
\ln \frac{\lambda(t_{0},f(t_{0}))}{ \lambda(t_{1},f(t_{1}))}=\int_{t_{1}}^{t_{0}}\frac{pa(1-n\rho)}{(-2a+2\rho a n) \tau+1}d\tau=\ln(\frac{-2a(1-n\rho ) t_{1}+1}{-2a(1-n\rho ) t_{0}+1})^{\frac{p}{2}}.
\end{equation*}
Since $\lambda(t_{0},f(t_{0}))=\lambda(t_{0})$ and  $\lambda(t_{1},f(t_{1}))\geq\lambda(t_{1})$ we conclude that
\begin{equation*}
\ln \frac{\lambda(t_{0})}{ \lambda(t_{1})}>\ln(\frac{-2a(1-n\rho ) t_{1}+1}{-2a(1-n\rho ) t_{0}+1})^{\frac{p}{2}},
\end{equation*}
 that is the quantity  $\lambda(t)[-2a(1-n\rho ) t+1]^{\frac{p}{2}}$ is strictly increasing   along the Ricci-Bourguignon  flow on $[0,T)$.
\end{example}


\end{document}